          \documentclass[a4paper,12pt,leqno]{article}
\usepackage{amsmath}%
\usepackage{amsfonts}%
\usepackage{amssymb}%
\usepackage{graphicx}
\newtheorem{theorem}{Theorem}

\newtheorem{definition}[theorem]{Definition}
\newtheorem{example}[theorem]{Example}

\newtheorem{lemma}[theorem]{Lemma}

\newtheorem{problem}[theorem]{Problem}

\newtheorem{remark}[theorem]{Remark}

\newenvironment{proof}[1][Proof]{\textbf{#1.} }{\ \rule{0.5em}{0.5em}}

\begin{document}

%\title{A note on Weierstrass points and Proof of Weierstrass gap theorem  }
\title{A note on meromorphic functions on a compact Riemann surface having poles at a single point}
\author{Gollakota V V Hemasundar}
%\thanks{This work was supported by The National Board for Higher Mathematics (NBHM) Govt. of India. Grant No.  02011/6/2023 NBHM(R.P)/4301.}
%\\Professor, Department of Mathematics, %\\SIWS College, Affiliated to University of Mumbai, \\Mumbai, India
%\\email: gvvhemasundar@gmail.com }%\\ ORCID: 0000-0002-0089-9816 }
\date{}
\maketitle

\begin{abstract}
The Riemann -Rock theorem plays a central role in the theory of Riemann surfaces with applications to several branches in Mathematics and Physics. Suppose $X$ ia a compact Riemann surface of genus $g$ and $P \in X$.  By the Riemann-Roch theorem  there exists a  meromorphic function on $X$  having  a pole at $P$ and is holomorphic in  $X \setminus \{P\}$.  The Weierstrass gap  theorem gives more information on the order of the pole at $P$. It  determines a sequence of $g$ distinct numbers $1 < n_k < 2g$, $1 \leq k \leq g$  for which a meromorphic function with the order  $n_k$, fails to exist at $P$ and it can be obtained again as an application of Riemann-Roch theorem. 
In this  note, we give proof of the Weierstrass gap theorem,  using the dimensions of the  cohomology groups  and find an interesting combinatorial problem, which may be seen as a byproduct from the statement of the Weierstrass gap theorem. A short note is given at the end on Weierstrass points where a meromorphic function with lower order  pole $\leq g$ exists and obtain some consequences of Weierstrass gap theorem.
\end{abstract}
AMS:2020, 30F10

\maketitle

\section{Introduction}
Let $X$ be a compact Riemann surface of genus $g$ and $P \in X$. By the Riemann-Rock theorem there exists a  non-constant meromorphic function $f \in \mathcal{M}(X)$ which has a pole of order $\leq g +1$ at $P$ and is holomorphic in $X \setminus \{P\}$. %If

The bound $g+1$ can be further improved to $g$ if the point $P$ is a Weierstrass point \cite{forster}. 

In this note we prove the following theorem which provides more information on the order of a pole at a point $P$ on $X$: 
\begin{theorem} [Weierstrass gap theorem]\label{wgp}  Let X be a compact Riemann surface, of  genus $g \geq 1$.  Suppose $P$ is a point on $X$. Then there are precisely $g$ integers, $n_k$ such that
\begin{equation}\label{wgs}
1=n_1  <  n_2 < \dots     <n_g <2g
\end{equation}  
such that there does not exist a meromorphic function on $X$ with a pole of order $n_k$ at $P$.
\end{theorem}
The numbers $n_k$, for $k = 1, \dots, g$ are called gaps at $P$ and their complement in $\mathbb{N}$ are called non-gaps.  There are precisely $g$ non-gaps in $\{2, \dots , 2g\}$.  Further, the gap sequence is uniquely determined by the point $P$. 

The cases $g = 0$ and $g=1$ can also be seen as part of the gap sequence. If $g=0$, there is always a meromorphic function function on the sphere with a pole at $P$.
In case of $g=1$, there is no meromorphic function on $X$ with a simple pole at $P$ and is holomorphic in $X \setminus \{P\}$, since there is no doubly periodic meromorphic function having precisely a simple pole at a single point in any period parallelogram.
 
In this note, Theorem \ref{wgp}  will be proved with the spirit of the proof given in \cite{springer}, using the terminology of sheaf cohomology.% \cite{gvvh}. 
A proof of the Weierstrass-gap theorem can be derived as a special case of Noether theorem. See Farkas and Kra \cite{farkas}. The proofs are generally given as an application of the Riemann-Roch Theorem. % \cite{springer}.

In section 2, we mention some of the consequences of Riemann-Roch and Serre Duality theorems \cite{forster} that are useful for comparing the dimensions of cohomology groups. The Proof of the Weierstrass-gap theorem is given in Section 3. In Section 4, we form a combinatorial problem that appears to be a byproduct of the statement of the Weierstrass gap theorem and mention some remarks on Weierstrass points. We conclude with Section 5.

\section{Some Consequences of Riemann-Roch and Serre Duality Theorems}

\begin{definition}
Let $D$ be a divisor on $X$. For any open set $U \subset X$, we define $\mathcal{O}_D(U)$ the set of all those meromorphic functions which are multiples of the divisor $-D$, i.e.,
$$
\mathcal{O}_D(U) := \{f \in \mathcal{M}(U) : ord_x(f) \geq - D(x) \mbox{ for every } x \in U\}
$$
$$
= \{f \in \mathcal{M}(U) : ord_x(f)+D(x)  \geq 0  \mbox{ for every } x \in U\}
$$
\end{definition}
One can verify with the natural restriction mappings, $ \mathcal{O}_D$ is a sheaf.   It denotes the sheaf of meromorphic functions which are multiples of the divisor $-D$. See \cite{forster}. 
 If $D=0$,  one has $\mathcal{O}_0 =\mathcal{O}$ which is the set of holomorphic functions on $X$.

To $D$ we can associate a line bundle $L_D$  such that the sheaf of holomorphic cross sections of $L_D$ is isomorphic to the sheaf $\mathcal{O}_D$ of  meromorphic multiples of $-D$.
\begin{theorem} \label{wgt1}
Let $X$ be a compact Riemann surface and $D\in Div(X)$. If $deg D < 0$ then $H^0\left(X, \mathcal{O}_D\right) = 0$.
\end{theorem}
\begin{proof}
Suppose $f \in H^0\left(X, \mathcal{O}_D\right)$ with $f \neq 0$.  Then $(f) \geq -D$ and thus 
$$
deg (f) \geq - deg D >0
$$
This contradicts the fact that $X$ is compact and $deg (f) = 0$.
\end{proof}

We can use Serre-Duality theorem to obtain equality of dimensions: 
\begin{equation}\label{sddimensions}
\dim H^0\left(X, \Omega_{-D} \right) =  \dim H^1\left(X, \mathcal{O}_D\right)
\end{equation}
For $D=0$ we obtain
\begin{equation}\label{genus}
\dim H^0\left(X, \Omega \right) = g= \dim H^1\left(X, \mathcal{O}\right)
\end{equation}
Here $H^0\left(X, \Omega \right) = \Omega(X)$ which denotes the sheaf of holomorphic 1-forms on $X$.
The following equation  can also be obtained as an application of Seree Duality Theorem:
\begin{equation}
\dim H^1\left(X, \Omega \right) = \dim  H^0\left(X, \mathcal{O} \right)  =1
\end{equation}
\begin{remark}
This is a known result that there are no non-constant holomorphic maps on a compact Riemann surface.
\end{remark}
Suppose $X$ is a compact Riemann surface of genus $g$. The divisor of any non-vanishing meromorphic 1-from $\omega$ on $X$ satisfies $deg(\omega) = 2g -2.$ We denote it by  $K $.  Hence the degree of a canonical divisor $K = 2 g - 2$.

We use the following form of Riemann Roch theorem and prove theorem \ref{wgp} as an application. % of the Riemann-Roch theorem which is 
\begin{equation}\label{thm rr}
 \dim H^0\left(X, \mathcal{O}_D\right) -  \dim H^0\left(X, \Omega_{-D} \right) = 1 - g + deg D
\end{equation}
\begin{lemma} \label{lemmah10}
Suppose $X$ is a compact Riemann surface of genus $g$ and $D$ is a divisor on $X$.  Then 
\begin{equation}
H^0\left(X, \Omega_{-D}\right) = 0   \mbox{ whenever } deg D > 2g-2
\end{equation}
\end{lemma}
\begin{proof}
Suppose $\omega$ is a non-vanishing meromorphic 1-form on $X$ and $K$ is its canonical divisor.  Then there is an isomorphism $\Omega_{-D} \cong \mathcal{O}_{K-D}.$  Thus 
$$
H^0\left(X, \Omega_{-D}\right) \cong H^0\left(X, \mathcal{O}_{K-D}\right) 
$$
If deg $D > 2g-2$, then deg $(K-D)<0$.\\  
Thus by Theorem (\ref{wgt1}) $H^0\left(X, \mathcal{O}_{K-D}\right) =0.$
\end{proof}
\section{Proof of Theorem \ref{wgp} (Weierstrass Gap Theorem)}
\begin{proof}
Suppose $P \in X$.  If $D$ is a zero divisor, then 
$ \dim H^1(X, \mathcal{O}) = g$  and $deg (D)=0$. 

By the Riemann Roch theorem $ \dim H^0(X, \mathcal{O}) = 1$.  Therefore, there are no non-constant holomorphic functions on $X$.\\
 Define the divisor $D_n$ such that 
$$
D_n(x) =
\begin{cases}
n  \mbox{ if } x = P \\
0  \mbox{ if } x \neq P
\end{cases}
$$
For $n=1$, we have
$Deg(D_1)=1$.  Once again by the Riemann -Roch theorem, %\ref{thm rr}
$$
\dim H^0(X, \mathcal{O}_{D_1}) = 2-g + \dim H^0(X, \Omega_{-D_1}).
$$
If $\dim H^0(X, \Omega_{-D_1}) = g$, then $\dim H^0(X, \mathcal{O}_{D_1}) = 2$, hence there exists $f \in \mathcal{M}(X)$ which has a simple pole at $P$ and is holomorphic in $X\setminus \{P\}$. 
 
If $\dim H^0(X, \Omega_{-D_1}) = g-1$, then $\dim H^0(X, \mathcal{O}_{D_1}) = 
1$,  hence there is no meromorphic function which has a simple pole at $P$ and is  holomorphic in $X\setminus \{P\}$.  

Now we want to see the effect of changing the divisor from $ D_{n-1}$  to $D_n.$   By the Riemann - Roch Theorem

$$
\dim H^0  (X, \mathcal{O}_{D_{n-1}}) = n-g + \dim H^0(X, \Omega_{-D_{n-1}})
$$
and 
$$
\dim H^0(X, \mathcal{O}_{D_{n}}) = n+1-g + \dim H^0(X, \Omega_{-D_{n}}) 
$$
If $$\dim H^0(X, \Omega_{-D_{n-1}}) = \dim H^0(X, \Omega_{-D_{n}}) $$ then

 $$\dim H^0(X, \mathcal{O}_{D_{n-1}}) + 1= \dim H^0(X, \mathcal{O}_{D_{n}}) .$$
So there exists a meromorphic function $f \in \mathcal{M}(X)$ with a pole of order $n$ at $P$ and is holomorphic in $X\setminus \{P\}$.

\noindent
If $$\dim H^0(X, \Omega_{-D_{n-1}}) = \dim H^0(X, \Omega_{-D_{n}})+1 $$ then
$$\dim H^0(X, \mathcal{O}_{{D_{n}}}) = \dim H^0(X, \mathcal{O}_{D_{n-1}}) .$$ 
So there will not exist a function with a pole of order $n$ at $P$ and is holomorphic in  $X\setminus \{P\}$.

So if $\dim H^0(X, \Omega_{-D_{n}})$ remains the same as $n$ increases by $1$, then a new linearly independent function is added in going from  the sheaf $\mathcal{O}_{D_{n-1}}$ to $\mathcal{O}_{D_{n}}$.

From the Equation  (\ref{genus}),  we have $ \dim H^0(X, \Omega) = g $.  Suppose $\omega$ is  a non-vanishing meromorphic 1-form on a compact Riemann surface of genus $g$, and $K$ is its divisor, then deg $(\omega) = 2g-2$.

By Lemma (\ref{lemmah10})
 $$\dim H^0(X, \Omega_{-D_{2g-1}})= 0.$$ 
Therefore, the number of times $\dim H^0(X, \Omega_{-D_{n}})$ does not remain the same must be $g$ times and at each change it decreases by 1. \\
This completes the proof.
\end{proof}

\section {Analyzing  gaps and non-gaps}
  
For details on gap sequence see \cite{farkas}. 
\begin{example}
Suppose $X$ is a compact Riemann surface fo $g=3$. We see the possible gap sequences  are\\
 $\{1,3,5\}$, $\{1, 2,3\}$, $\{1,2,4\}$, $\{1,2,5\}$\\
and corresponding non-gap sequences  in $\{2, \dots , 2g\}$  are:\\ $\{2,4,6\}$, $\{4, 5,6\}$, $\{3,5,6\}$, $\{3,4,6\}$. \\
Note that $2g$ is always a non-gap. 
\end{example}
Suppose $P \in X$. 
If $f$ has a pole of order $s$ at $P$, and $g$ has a pole of order $t$ at $P$, then $fg$ has a pole of order $s+t$ at $P$. 
The following problem may be of some interest to see the number of possible gaps and non-gaps $\leq 2g$ on $X$. 
The reader must note that this problem is  nothing to do with the proof of the theorem.
\noindent
\begin{problem}
 Write the numbers  $2$ to $2g-1$ into two (disjoint) parts $G =\{n_1, n_2, \dots , n_{g-1}\}$ and $G'= \{m_1, m_2, \dots,
 m_{g-1} \}$ such that no  number in $G$ is   a sum of any combination of numbers in $G'$. How many pairs of such $G$ and $G'$ exist?

Clearly then $\{1\} \cup G$ gives possible gap sequence and $G' \cup \{2g\}$ gives corresponding  non-gap sequence in $\{2, \dots , 2g\}$ at a point $P$ on $X$.
\end{problem}

We define a Weierstrass point by using gap sequence as follows:

\begin{definition}
Suppose $ P \in X$ and
$$
0 < n_1 < n_2 < \dots n_g <2g $$  be the gap sequence at $P$. 
 In terms of the gap sequence we define the weight of the point $P$,  denoted by $\omega(P)$ and defined  by 
	$$
	\omega(P)=\sum_{i=1}^g (n_i - i)
	$$
\end{definition}
Note that $w(P) \geq 0$ for all $P \in X$. 
\begin{definition} \label{weierstrass point}
A point $ P\in X$ is called a Weierstrass point if $ \omega(P) >0$.    
\end{definition}
One can compute the number of Weierstrass points  counted according to their weights on a compact Riemann surface $X$ of genus $g$. It is equal to  $  g^3 - g = (g-1)g(g+1)$.

 It follows that there are no Weierstrass points on the surfaces of genus $g=0$ and genus $g=1$. Also from the theorem \ref{wgp}, it follows that there is no non-constant meromorphic function on torus ($g=1 $ surface) with a single simple(=order 1) pole.

The following theorem gives the bounds for Weierstrass points on a compact Riemann surface of genus $g\geq 2$.

\begin{theorem}
Suppose $X$ is a compact Riemann surface of genus $g \geq 2$.  let $W(X)$ denotes the number of Weierstrass points on $X$. Then 
\begin{equation}
2g+ 2 \leq W(X) \leq g^3 - g
\end{equation}
\end{theorem}
A point $p$ is called a hyperelliptic Weierstrass point if the non-gap sequence starts with 2 and the hyperelliptic Riemann surfaces are characterized by the gap sequence at the Weierstrass points: 
\begin{equation} \label{he}
P= \{1,3, \dots , 2g-1\} 
\end{equation}
hence the non-gaps are
  Q=\{2,4, \dots, 2g\}
	
Let $X$ be  a hyperelliptic Riemann surface and $p$ be a Weierstrass point on $X$. Then we can find 
\begin{equation}\label{he}
\omega(p) = [1 + 3 + \dots + (2g-1)] -  [ 1 + 2 + \dots + g] =  \frac {g (g-1)}{2}
\end{equation}
%It is known that the number of Weierstrass points on a compact Riemann surface of genus $g$ is equal to $g^3-g$.  See \cite{forster}.  
Therefore, from the Equation \ref{he} we can see that there are precisely $2g+2$ Weierstrass points on a hyperelliptic Riemann surface. 

\begin{remark}
In terms of Weierstrass points, the hyperelliptic Riemann surfaces may be characterized as the surfaces that attain the lower bound on the number of Weierstrass points.
\end{remark}

Jenkins proved that if $h$ is a first non-gap at $P$ and $(h,k)=1$, then  $k$ is a gap if $$g>\frac{(h-1)(k-1)}{2}.$$  Further he observes, hyperelliptic Riemann surfaces are  a special case whose first non-gap  sequence starts with a prime number.  See \cite{jen}. 

For the Exceptional Riemann surfaces the gap sequence $G$ and the non-gap sequence $G'$ at each Weierstrass point is given by
\begin{equation} \label{er}
G = \{1,2,3, \dots , g-1, g+1\}  \mbox { and }
G' = \{ g, g+2, \dots , 2g-1\}.
\end{equation}

\section{Concluding Remarks } \label{WP}
Some times Theorem  \ref{wgp} will give information on the order of a pole of a meromorphic function at a point $P$ and holomorphic in $X \setminus \{P\}$. The following example illustrates this:
\begin{example}
Suppose $X$ is a  hyperelliptic Riemann surface of genus 6 and $P$ is a Weierstrass point.  Then there exist meromorphic functions with order 2, 4 and 6  at $P$ which are non-gaps of hyperelliptic gap sequence.  But at the same time if $P$ is not a Weierstrass point, there is no meromorphic function with a  pole of order $<7$ at $P$, and holomorphic in $X \setminus \{P\}$.   
\end{example}

There is an another important application of Weierstrass gap sequence to prove Automorphism group of a compact Riemann surface of genus $g \geq 2$ is finite.  This is a well known result proved by Schwarz \cite{farkas}. %and Hurwitz found that the number of automorphisms does not exceed $84(g-1)$. 
The proof of finiteness of the $Aut(X)$ follows from the observation that if $\sigma $ is an automorphism of a compact Riemann surface of genus $g \geq 2$, then $\sigma ((W(X)) = W(X)$ where $W(X)$ denotes the set of Weierstrass points of $X$. This follows again because of the gap sequence at a point $P$ and at $\sigma(P)$ are the same.

\subsection*{Acknowledgments}
This work was supported by The National Board for Higher Mathematics (NBHM) Govt. of India. Grant No.  02011/6/2023 NBHM(R.P)/4301. % I 
\subsection*{Financial Interests}The author has no relevant financial or non-financial interests to disclose.

\end{document}